\documentclass{amsart}
\usepackage{amssymb}
\usepackage{tikz}

\DeclareMathOperator{\ad}{ad}

\DeclareMathOperator{\lie}{\mathfrak{lie}}

\DeclareMathOperator{\grt}{\mathfrak{grt}}

\DeclareMathOperator{\ch}{ch}

\newcommand\Lie[1]{\mathfrak{#1}}
\newcommand{\g}{\Lie{g}}

\renewcommand{\t}{\Lie{t}}

\renewcommand{\k}{\mathbb{K}}

\newcommand{\R}{\mathbb{R}}
\newcommand{\Z}{\mathbb{Z}}
\newcommand{\Q}{\mathbb{Q}}

\newcommand{\gt}{\mathfrak{gt}}

\theoremstyle{plain}
\newtheorem{theorem}{Theorem}[section]

\newtheorem{proposition}[theorem]{Proposition}

\theoremstyle{definition}

\newtheorem{remark}{Remark}[section]

\begin{document}

\title[]{On rational Drinfeld associators}

\author{Anton Alekseev}

\address{Section of Mathematics, University of Geneva, 2-4 rue du Li\`evre,
c.p. 64, 1211 Gen\`eve 4, Switzerland}

\email{Anton.Alekseev@unige.ch}

\author{Masha Podkopaeva}

\address{Section of Mathematics, University of Geneva, 2-4 rue du Li\`evre,
c.p. 64, 1211 Gen\`eve 4, Switzerland}

\email{Maria.Podkopaeva@unige.ch}

\author{Pavol \v Severa}
\address{Section of Mathematics, University of Geneva, 2-4 rue du Li\`evre,
c.p. 64, 1211 Gen\`eve 4, Switzerland, on leave from FMFI UK Bratislava, Slovakia}
\email{Pavol.Severa@gmail.com}

%\date{\today}

\begin{abstract}
We prove an estimate on denominators of rational Drinfeld associators.
To obtain this result, we prove the corresponding estimate for the $p$-adic associators stable under the action of suitable elements of $\text{Gal}(\bar{\mathbb{Q}}/\mathbb{Q})$.
As an application, we settle in the positive Duflo's question on the
Kashiwara--Vergne factorizations of the Jacobson element $J_p(x,y)=(x+y)^p-x^p-y^p$
in the free Lie algebra over a field of characteristic $p$. Another application is a new estimate on  denominators of the Kontsevich knot invariant.
\end{abstract}

\subjclass{}

\maketitle

\section{Introduction}
Drinfeld associators were defined in \cite{Dr} and play an important
role in many fields of mathematics, including number theory
(for recent developments, see \cite{numbertheory}), low-dimensional
topology \cite{Bar-Natan 2, Bar-Natan, Le Murakami}, Lie theory \cite{EK}, and deformation quantization
\cite{Tamarkin}. In literature, there are
two examples of Drinfeld
associators defined by explicit formulas. The Knizhnik--Zamolodchikov
associator $\Phi_\textit{KZ}$ of \cite{Dr} is defined over $\mathbb{C}$ and
expressed in terms of iterated integrals and multiple zeta values.
An explicit example over $\mathbb{R}$ is given in \cite{AT'} and \cite{SW}
in terms of Kontsevich integrals over configuration spaces.

The existence of rational associators was proved in \cite{Dr}. Constructing
an explicit example is considered to be a major
problem in associator theory, and (to the best of our knowledge) it remains
open.

In this paper, we obtain estimates on denominators of rational associators. More
precisely, we define the set of {\em natural} rational associators
with the property that the denominator
in its degree $n$ component is a divisor of
\begin{equation}\label{D(n)}
D(n) =  \prod_{\substack{{p \text{ prime}}\\ p\leq n+1}}
p^{b_p(n)},
\end{equation}
where
$$
\qquad b_p(n)=\Big[\frac{p}{(p-1)^2}\,n-\frac{1}{p-1}\Big],
$$
and $[\alpha]$ denotes the integer part of $\alpha\in\mathbb{R}$. Natural associators  have non-zero convergence
radius in the $p$-adic norm for every prime $p$.

 Our main result is Theorem \ref{maintheorem},
which shows that the set of natural associators is nonempty.
Our strategy for proving Theorem \ref{maintheorem} is as follows. There is a classical
Galois theory result of Drinfeld stating that the Grothendieck--Teichm\"uller
group character $\chi: {\rm GT}_p \rightarrow \mathbb{Z}_p^*$ is surjective.
If we choose an element $g_p\in{\rm GT}_p$ such that $\chi(g_p)$ generates a dense subgroup of $\mathbb{Z}_p^*$, then the unique $p$-adic associator fixed by $g_p$ satisfies the required denominator estimates. We use these $p$-adic associators for all primes $p$ to prove existence of a rational associator with the same estimates on denominators.

For a natural associator, the prime $p>2$ appears for the first time in the
denominator of the component of degree $p-1$. We show that this part of the
estimate is optimal. We use this observation to settle in the positive
Duflo's question \cite{Duflo} on Kashiwara--Vergne factorizations of
the Jacobson element. 

In more detail, the Kashiwara-Vergne problem \cite{KV} in
Lie theory is to find a factorization of the Campbell-Hausdorff series
in the form
$$
x+y-\ln(e^ye^x)=(1-\exp(-\ad_x)) A(x,y)+ (\exp(\ad_y)-1) B(x,y),
$$
where $A(x,y)$ and $B(x,y)$ are Lie series which satisfy an additional linear equation
with coefficients given by Bernoulli numbers (for details see {\em e.g.} \cite{bourbaki}).
The factor $p^{-1}$ appears for the first time in 
$\ln(e^ye^x)$ in degree $p$. The corresponding residue is the Jacobson element
$J_p(x,y)=(x+y)^p-x^p-y^p$. Hence, it is natural to conjecture that $A(x,y)$ and $B(x,y)$
are $p$-integral up to degree $p-2$ and have a simple pole in degree $p-1$.
If this is the case, one obtains a decomposition of the Jacobson element 
$$
J_p(x,y)=[x, a(x,y)] + [y, b(x,y)] \mod p,
$$
where $a(x,y)$ and $b(x,y)$ are residues of $A(x,y)$ and $B(x,y)$, respectively.
Finding such decompositions (with an extra constraint similar to the one in the
Kashiwara-Vergne problem) is the Duflo's question.

Another application of our results is to the theory of knot invariants. In \cite{Kontsevich} Kontsevich constructed a universal finite
type invariant of knots in $\R^3$. For a knot $K$, this invariant is denoted $I(K)$, and it takes values
in the graded algebra of chord diagrams ${\rm Chord}(\mathbb{Q})$. There is a combinatorial construction
of $I(K)$ which uses an arbitrary Drinfeld associator \cite{Bar-Natan 2, Le Murakami}. It is remarkable that the final result 
is independent of the associator used in the computation. In particular, by choosing a
natural rational associator we obtain a new estimate (Theorem \ref{something}) on
denominators of $I(K)$,
$$
I(K) \in \sum_{n=0}^\infty  \, D(n)^{-1} 
{\rm Chord}_n(\mathbb{Z}).
$$
This improves the estimate of \cite{Thang Le} (where the analogue
of $b_p(n)$ is quadratic in $n$).

In Appendix we prove a denominator estimate for the elements $(1,\psi)$ of the Grothendieck--Teichm\"uller Lie algebra $\mathfrak{gt}(\mathbb{Q}_p)$.

\subsection*{Acknowledgements}
 We are grateful to M. Duflo for introducing us to the
factorization problem for the Jacobson element. This problem played a key role
in our study. We would like to thank D. Bar-Natan, B. Enriquez and H. Furusho for interesting
discussions. Our research was suppoted in part by grants 200020-126817
and 200020-120042 of the Swiss National Science Foundation.

\section{The Grothendieck--Teichm\"uller group}

\subsection{Groups ${\rm GT}$ and ${\rm GRT}$}
Let $\k$ be a field of characteristic zero, and let
$\lie_n(\k)=\lie(x_1,\dots, x_n;\k)$ be the degree completion
of the free Lie algebra over $\k$ with generators $x_1, \dots, x_n$.
For instance, the Campbell--Hausdorff series
$\ch(x_1, \dots, x_n)=\ln(e^{x_1}\dots e^{x_n})$ is an element
of $\lie_n(\k)$.

Let $F_n$ be a free group with generators $X_1, \dots, X_n$,
and ${\rm PB}_n$ be the pure braid group for $n$ strands with standard
generators  $X_{i,j}$ for $i<j$ (the strand $i$ makes a tour around
the strand $j$). Denote by $F_n(\k)$ and ${\rm PB}_n(\k)$ their
$\k$-prounipotent completions. By putting $x_i=\ln(X_i)$ one recovers
the free Lie algebra $\lie_n(\k)$ with generators $x_1, \dots, x_n$.
Similarly, by putting $x_{i,j}=\ln(X_{i,j})$
one obtains the (filtered) Lie algebra $\mathfrak{pb}_n(\k)$ with
generators $x_{i,j}$ for $1 \leq i < j \leq n$. The associated graded
Lie algebra is the Lie algebra $\mathfrak{t}_n(\k)$ of infinitesimal braids
with generators $t_{i,j}=t_{j,i}$ for $i,j=1,\dots,n$, $i\neq j$ and
relations $[t_{i,j}, t_{i,k}+t_{j,k}]=0$ for all triples $i,j,k$
and $[t_{i,j}, t_{k,l}]=0$ for distinct $i,j,k$, and $l$.

For a commutative ring $R$, we will denote by $R\langle x_1, \dots, x_n
\rangle^k$ the $R$-module spanned by homogeneous non-commutative polynomials
of degree $k$ with coefficients in $R$, by $R\langle x_1, \dots, x_n
\rangle^{\leq k}$ the module spanned by non-commutative polynomials of degree
at most $k$, and by $R\langle\!\langle x_1, \dots, x_n \rangle\!\rangle^{\geq k}$
the module spanned by non-commutative formal power series of degree at
least $k$. Recall that one can view $F_n(\k)$ as the set of group-like
elements in $\k\langle\!\langle x_1, \dots, x_n \rangle\!\rangle$ equipped
with the standard co-product $\Delta(x_i)=x_i \otimes 1 + 1 \otimes x_i$.

The Grothendieck--Teichm\"uller group ${\rm GT}(\k)$ is defined in \cite{Dr}
as the set of pairs $(\lambda, f)$ with $\lambda \in \k^*$
and $f \in F_2(\k)$ such that
\begin{equation} \label{GT1}
f(x,y)=f(y,x)^{-1},
\end{equation}
\begin{equation}  \label{GT2}
f(z,x)e^{mz} f(y, z) e^{my} f(x, y) e^{mx} =1,
\end{equation}
where $e^xe^ye^z=1$ and $m=(\lambda-1)/2$, and
\begin{multline}  \label{GT3}
f(x_{1,2}, \ch(x_{2,3},x_{2,4}))f(\ch(x_{1,3},x_{2,3}), x_{3,4})=  \\
f(x_{2,3},x_{3,4})f(\ch(x_{1,2},x_{1,3}), \ch(x_{2,4},x_{3,4}))f(x_{1,2},x_{2,3}).
\end{multline}
The last equation is understood as an equality in ${\rm PB}_4(\k)$.
The group law of ${\rm GT}(\k)$ is defined by the formula
$(\lambda_1, f_1) \cdot (\lambda_2, f_2)=(\lambda, f)$, where
$\lambda=\lambda_1\lambda_2$ and
\begin{equation}  \label{grtgroup}
f(x,y)=f_1(\lambda_2 f_2(x,y)xf_2(x,y)^{-1}, \lambda_2 y)f_2(x,y).
\end{equation}
We denote by $\chi: {\rm GT}(\k) \rightarrow \k^*$ the group
homomorphism defined by the formula $\chi(\lambda, f)=\lambda$.

The Lie algebra $\mathfrak{gt}(\k)$ (corresponding to the group ${\rm GT}(\k)$)
is the set of pairs $(s, \psi)$ with $s\in \k$ and
$\psi\in \lie_2(\k)$ such that
\begin{equation}  \label{gt1}
\psi(x,y)=-\psi(y,x),
\end{equation}
\begin{equation}   \label{gt2}
\psi(x,y)+\psi(y,z)+\psi(z,x)+
\frac{s}{2} (x+y+z)=0,
\end{equation}
for $\ch(x,y,z)=0$ (i.e., $z=-\ch(x,y)$), and
\begin{multline}  \label{gt3}
\psi(x_{1,2}, \ch(x_{2,3}, x_{2,4})) + \psi(\ch(x_{1,3}, x_{2,3}), x_{3,4})= \\
\psi(x_{2,3}, x_{3,4})+ \psi(\ch(x_{1,2}, x_{1,3}), \ch(x_{2,4},x_{3,4}))
+\psi(x_{1,2},x_{2,3}).
\end{multline}
The last equation is understood as an equality in the Lie algebra
$\mathfrak{pb}_4(\k)$.

The kernel of the Lie homomorphism $\chi: (s, \psi) \mapsto s \in \k$
is the Lie subalgebra $\mathfrak{gt}_1(\k)$. It admits a graded version
$\mathfrak{grt}(\k)$ formed by elements $\psi \in \mathfrak{lie}_2(\k)$ satisfying
equation \eqref{gt1}, equation
\begin{equation}  \label{grt2}
\psi(x,y)+\psi(y,z)+\psi(z,x)=0,
\end{equation}
for $x+y+z=0$, and
\begin{multline}  \label{grt3}
\psi(t_{1,2}, t_{2,3}+t_{2,4}) + \psi(t_{1,3}+t_{2,3}, t_{3,4})= \\
\psi(t_{2,3}, t_{3,4})+ \psi(t_{1,2}+t_{1,3}, t_{2,4}+t_{3,4})
+\psi(t_{1,2},t_{2,3}).
\end{multline}
Here the grading is induced by the natural grading of $\mathfrak{lie}_2(\k)$.
The corresponding group is denoted by ${\rm GRT}(\k)$. This group is equipped with the group law
$f_1 \cdot f_2 =f$, where $f(x,y)=f_1(f_2(x,y)\,x\,f_2(x,y)^{-1},y)f_2(x,y)$.

\subsection{The group ${\rm GT}$ over $\mathbb{Q}_p$}
Let $p>2$ be a prime, $\mathbb{Q}_p$ be the field of rational
$p$-adic numbers, and $\Z_p$ be the ring of $p$-adic integers. Consider the subgroup
${\rm GT}_p \subset {\rm GT}(\mathbb{Q}_p)$ that consists of the
pairs $(\lambda, f)$ with $\lambda \in \mathbb{Z}_p^*$ and
$f \in (F_2)_p$, where $(F_2)_p$ is the pro-$p$ completion of the
free group $F_2$. Recall \cite{Lazard} that the elements of $(F_2)_p$
can be viewed as group-like elements in
$
\mathbb{Z}_p\langle\!\langle \hat{x},\hat{y} \rangle\!\rangle \subset
\mathbb{Q}_p\langle\!\langle {x},{y} \rangle\!\rangle
$,
where $\hat{x}=e^x-1,$ $\hat{y}=e^y-1$, and the standard co-product
is given on generators by $\Delta(\hat{x})=1\otimes \hat{x}+
\hat{x}\otimes 1+\hat{x}\otimes \hat{x},$ $\Delta(\hat{y})=1\otimes \hat{y}+
\hat{y}\otimes 1+\hat{y}\otimes \hat{y}$.

We shall need the following classical result which follows  from surjectivity of the $p$\nobreakdash-adic cyclotomic character $\text{Gal}(\bar{\mathbb{Q}}/\mathbb{Q})\to\mathbb{Z}_p^*$ and from  the fact that this character is the composition of a group morphism $\text{Gal}(\bar{\mathbb{Q}}/\mathbb{Q})\to\text{GT}_p$ with the character $\chi: {\rm GT}_p \rightarrow \mathbb{Z}_p^*$
(see the proof of Proposition 5.3 in \cite{Dr}):

\begin{theorem}
The character $\chi: {\rm GT}_p \rightarrow \mathbb{Z}_p^*$ is surjective.
\end{theorem}

\section{$p$-adic associators}

\subsection{Drinfeld associators}
For $\mu \in \k$, the set of Drinfeld associators ${\rm Assoc}_\mu(\k)$
is defined as the set of group-like
elements $\Phi \in \k\langle\!\langle x,y \rangle\!\rangle$ such that
\begin{equation}  \label{assoc1}
\Phi(x,y)=\Phi(y,x)^{-1},
\end{equation}
\begin{equation}  \label{assoc2}
e^{\mu x/2} \Phi(z,x) e^{\mu z/2} \Phi(y,z) e^{\mu y/2} \Phi(x,y)=1
\end{equation}
for $x+y+z=0$, and
\begin{multline}  \label{assoc3}
\Phi(t_{1,2}, t_{2,3}+t_{2,4})\,\Phi(t_{1,3}+t_{2,3}, t_{3,4})=\\
\Phi(t_{2,3}, t_{3,4})\, \Phi(t_{1,2}+t_{1,3}, t_{2,4}+t_{3,4}) \,
\Phi(t_{1,2}, t_{2,3})
\end{multline}
in the group $\exp(\t_4)$.

\begin{remark}
It has recently been shown by Furusho \cite{Furusho} that equations
\eqref{assoc1} and \eqref{assoc2} are implied by the pentagon
equation \eqref{assoc3} and by $\Phi\in 1+\k\langle\!\langle x,y \rangle\!\rangle^{\geq2}$.
The coefficient $\mu$ in \eqref{assoc2} is determined by the expansion
$\Phi(x,y)=1+\mu^2/24\, [x,y] + \dots$, where the dots stand for terms
of degree higher than two.
\end{remark}

Let
$${\mathrm{Assoc}}(\k)=\{(\mu,\Phi);\mu\in\k^*,\Phi\in\mathrm{Assoc}_\mu(\k)\}$$
be the set of all associators with $\mu\neq0$ (${\mathrm{Assoc}}_0(\k)$ coincides with $\mathrm{GRT}(\k)$).
The group ${\mathrm{GT}}(\k)$ acts on ${\mathrm{Assoc}}(\k)$ via
$$(\lambda,f)\cdot(\mu,\Phi)=(\lambda\mu,f(\mu\,\Phi(x,y)\, x\, \Phi(x,y)^{-1},\, \mu\,y)\,\Phi(x,y)).$$
This action is free and transitive.

The group $\k^*$ acts on $\mathrm{Assoc}(\k)$ by $\lambda\cdot(\mu,\Phi)=(\lambda\mu,\Phi(\lambda x,\lambda y))$. By identifying the quotient $\mathrm{Assoc}(\k)/\k^*$ with $\mathrm{Assoc}_1(\k)$,
we get an action of $\mathrm{GT}(\k)$ on $\mathrm{Assoc}_1(\k)$.
Under this action, every $\Phi\in\mathrm{Assoc}_1(\k)$ has a one-parameter stabilizer of
the form $\{ (\lambda, f_\lambda) \in {\rm GT}, \lambda \in \k^*\}$, i.e.,

\begin{equation}\label{flambdaPhi}
f_\lambda(\Phi(x,y) x \Phi(x,y)^{-1}, y)\,\Phi(x,y) = \Phi(\lambda x, \lambda y).
\end{equation}
This one-parameter subgroup  of ${\rm GT}(\k)$ is generated by an element
$(1,\psi) \in \mathfrak{gt}(\k)$ satisfying
\begin{equation}   \label{psiphi}
\Phi(x,y)^{-1} \, \frac{d}{d\lambda}  \Phi(\lambda x,\lambda y)|_{\lambda=1} =
\psi(x, \Phi(x,y)^{-1}y\,\Phi(x,y)).
\end{equation}
Equation \eqref{psiphi} gives a bijective correspondence between the elements $\psi$ of $\k\langle\!\langle x,y \rangle\!\rangle^{\geq1}$ and the elements $\Phi$ of
$1+\k\langle\!\langle x,y \rangle\!\rangle^{\geq1}$. By Proposition 5.2 in \cite{Dr},
this correspondence restricts to a bijection
between  ${\rm Assoc}_1(\k)$ and the set of elements
$(s, \psi) \in \mathfrak{gt}(\k)$ with $s=1$.

\begin{theorem}\label{thm:GT-to-assoc}
Let $(\lambda_0,f)\in \mathrm{GT}(\k)$ be such that $\lambda_0\in\k^*$ is not a root of unity. 
Then there is a unique associator $\Phi\in\mathrm{Assoc}_1(\k)$ such that 
$$(\lambda_0,f)\cdot(1,\Phi)=\lambda_0\cdot(1,\Phi),$$
i.e. such that
\begin{equation}\label{f0Phi}
f(\Phi(x,y) x \Phi(x,y)^{-1}, y)\,\Phi(x,y) = \Phi(\lambda_0 x, \lambda_0 y).
\end{equation}
\end{theorem}
\begin{proof}
Since $\lambda_0$ is not a root of 1, there is a unique element $\Phi\in 1+\k\langle\!\langle x,y \rangle\!\rangle^{\geq1}$ satisfying \eqref{f0Phi}. Indeed, the degree $n$ homogeneous part of \eqref{f0Phi} is of the form
\[\Phi_n+Q_n(x,y,\Phi_1,\dots,\Phi_{n-1})=\lambda_0^n \Phi_n,\]
where $Q_n$ is a non-commutative polynomial and $\Phi_n$ is the degree $n$ homogeneous part of $\Phi$. We thus have a recurrence relation
\[\Phi_n=\frac{1}{\lambda_0^n-1}\,Q_n(x,y,\Phi_1,\dots,\Phi_{n-1}), \]
which admits a unique solution.

We need to prove that $\Phi\in\rm{Assoc}_1(\k)$. Notice that for each $\lambda \in \k$ there is a unique $f_\lambda\in\k\langle\!\langle x,y \rangle\!\rangle$
 satisfying \eqref{flambdaPhi}, i.e.\ such that $(\lambda,f_\lambda)\cdot(1,\Phi)=\lambda\cdot(1,\Phi)$. Moreover, the degree $n$ part of $f_\lambda$ is a polynomial in $\lambda$
 of degree at most $n$.
By the uniqueness property, we have $(\lambda,f_\lambda)|_{\lambda=\lambda_0}=(\lambda_0,f)$ and more generally $(\lambda,f_\lambda)|_{\lambda=\lambda_0^n}=(\lambda_0,f)^n$ for every $n\in\mathbb{Z}$. 
Since $(\lambda,f_\lambda)\in\rm{GT}(\k)$ for infinitely many values of $\lambda$ (namely for $\lambda=\lambda_0^n$), we have $(\lambda,f_\lambda)\in\rm{GT}(\k)$ for every non-zero value of $\lambda$.

Define $h_\epsilon(x,y)=f_{1/\epsilon}(\epsilon x , \epsilon y)$, its coefficients are polynomials in $\epsilon$.
By making a substitution $\lambda =1/\epsilon$ and $x \mapsto \epsilon x, y \mapsto \epsilon y$ in Equation \eqref{flambdaPhi} we get 
$$
\Phi(x,y)=h_\epsilon(\Phi(\epsilon x, \epsilon y) x \Phi(\epsilon x, \epsilon y)^{-1}, y)\Phi(\epsilon x, \epsilon y) ,
$$ 
and hence $\Phi(x,y)=h_0(x,y)$. Relations \eqref{assoc1}, \eqref{assoc2}, \eqref{assoc3} for $\Phi$ then follow from relations \eqref{GT1}, \eqref{GT2}, \eqref{GT3} for $(\lambda,f_\lambda)$.
Therefore $\Phi\in\rm{Assoc}_1(\k)$.
\end{proof}

The group ${\rm GRT}(\k)$ acts on $\mathrm{Assoc}_1(\k)$ by
$$
g\cdot \Phi  =  \Phi(g(x,y)\, x\, g(x,y)^{-1}, y)\, g(x,y).
$$
This action is free and transitive.

Following Drinfled, we introduce the set of associators
${\rm Assoc}_1^{(k)}(\k) \subset \k\langle x,y \rangle^{\leq k} \cong
\k\langle\!\langle x,y \rangle\!\rangle/\k\langle\!\langle x,y \rangle\!\rangle^{\geq k+1} $
satisfying equations \eqref{assoc1}, \eqref{assoc2} (with $\mu=1$), and \eqref{assoc3}
up to degree $k$. By Proposition 5.8 in \cite{Dr}, the natural projections
$ {\rm Assoc}_1^{(k+1)}(\k) \rightarrow {\rm Assoc}_1^{(k)}(\k)$ are surjective for all $k$.
One can also introduce Lie algebras $\mathfrak{gt}^{(k)}(\k) \subset \k\langle x,y \rangle^{\leq k}$
which consist of the pairs
$(s, \psi)$ with
$s \in \k,$ $\psi \in \mathfrak{lie}_2(\k) \cap\k\langle x,y \rangle^{\leq k}$
satisfying equations
\eqref{gt1}, \eqref{gt2} and \eqref{gt3} up to degree $k$. We will need the following
version of Proposition 5.2 in \cite{Dr}.

\begin{proposition}  \label{uptodegreek}
For all $k$, formula \eqref{psiphi} gives a one-to-one correspondence
between the elements $(1,\psi) \in \mathfrak{gt}^{(k)}(\k)$ and the elements $\Phi \in {\rm Assoc}_1^{(k)}(\k)$.
\end{proposition}

\begin{proof}
Let $(1,\psi) \in \mathfrak{gt}^{(k)}(\k)$. By Proposition 5.6 in \cite{Dr}, the Lie algebra
$\mathfrak{gt}(\k)$ is isomorphic (in a non-canonical way) to the semi-direct sum
$\k \oplus \mathfrak{grt}(\k)$. We choose such an isomorphism $\tau:\mathfrak{gt}(\k) \rightarrow
\k \oplus \mathfrak{grt}(\k)$, and apply it to $(1, \psi)$ to obtain an element
$1 + \phi \in \k \oplus \mathfrak{grt}^{(k)}(\k)$. Since $\mathfrak{grt}(\k)$ is graded,
the element $1 + \phi$ lifts (in many ways) to an element $1 + \hat{\phi} \in
\k \oplus \mathfrak{grt(\k)}$. By applying the inverse of $\tau$, we arrive at
$(1, \hat{\psi}) \in \mathfrak{gt}(\k)$ which is a lift of $(1, \psi)$. Consider the
associator $\hat{\Phi}$ corresponding to the element $(1, \hat{\psi})$, and let $\Phi$
be its degree $k$ truncation. Equation \eqref{psiphi} is verified for
$\hat{\psi}$ and $\hat{\Phi}$. By truncating it at degree $k$,
we can replace $\hat{\psi}$ by $\psi$ and $\hat{\Phi}$ by $\Phi$.

Similarly, let $\Phi \in {\rm Assoc}_1^{(k)}(\k)$, and let $\hat{\Phi} \in {\rm Assoc}_1(\k)$
be an associator lifting $\Phi$. We define $\hat{\psi}$ by equation \eqref{psiphi} to obtain
$(1, \hat{\psi}) \in \mathfrak{gt}(\k)$, and let $\psi$ be the degree $k$ truncation of $\hat{\psi}$.
Again, by truncating equation \eqref{psiphi} at degree $k$,
we can replace $\hat{\psi}$ by $\psi$ and $\hat{\Phi}$ by $\Phi$.
\end{proof}

\subsection{Denominator estimates of associators over $\mathbb{Q}_p$}
In this section we prove existence of $p$-adic associators with certain denominator bounds. The bounds are
given by a function $a_p$ defined as
\begin{equation}
\qquad a_p(n)=\Big[\frac{n}{p-1}\Big] + v_p\Big(\,\Big[\frac{n}{p-1}\Big]!\,\Big)\, ,
\end{equation}
where $v_p$ is the $p$-adic valuation.

\begin{proposition}  \label{prop:a}
The function $a_p(n)$ has the following properties
\begin{equation} \label{a1}
a_p(n) \geq a_p(k)+a_p(l) \text{ if }n\geq k+l ,
\end{equation}
\begin{equation} \label{a2}
a_p(m+n) \geq a_p(m+1) + v_p(n!) ,
\end{equation}
%
%\begin{equation} \label{a3}
%a_p(n) \geq a_p(m) + v_p(n)
%\end{equation}
%for $m \leq n-1$,
%%
%\begin{equation}  \label{a4}
%a_p(n) \geq a_p(m)+1+v_p(n)
%\end{equation}
%for $m \leq n+1-p$.
\end{proposition}

\begin{proof}
Inequality \eqref{a1} follows from $v_p\big((k+l)!\big)\geq v_p(k!)+v_p(l!)$ and inequality \eqref{a2} from $v_p(n!)\leq[(n-1)/(p-1)]$.
\end{proof}

%For the proof of Proposition \ref{prop:a}, see Appendix.

Property \eqref{a1} ensures that
$ \sum_{n=0}^\infty p^{-a_p(n)} \Z_p\langle x,y \rangle^n$
is a subring of $\mathbb{Q}_p\langle\!\langle x,y \rangle\!\rangle$, and property \eqref{a2} implies
$$
\sum_{n=0}^\infty p^{-a_p(n)} \Z_p\langle x,y \rangle^n =
\sum_{n=0}^\infty p^{-a_p(n)} \Z_p\langle \hat{x},\hat{y} \rangle^n,
$$
for $\hat{x}=e^x-1$, and $\hat{y}=e^y-1$.

For any prime $p>2$, we can choose an element $\lambda\in\mathbb{Z}_p^*$ generating a dense subgroup of $\mathbb{Z}_p^*$, i.e., such that $\lambda\mod p$\  is a generator of $\mathbb{F}_p^*$ and $\lambda^{p-1}\in 1+p\mathbb{Z}_p^*$.

\begin{theorem}  \label{Phiestimate}
Let $p$ be a prime.
For $p>2$, let $(\lambda, f) \in {\rm GT}_p$ be such that
$\lambda$ generates a dense sugroup of $\mathbb{Z}_p^*$, and, for $p=2$, let $(\lambda, f) \in {\rm GT}_2$
be such that $\lambda \in 1+4+8\Z_2$. Then there is a unique
Drinfeld associator $\Phi \in {\rm Assoc}_1(\mathbb{Q}_p)$ such that
\begin{equation}  \label{fPhi}
f(\Phi(x,y)\, x\, \Phi(x,y)^{-1}, y)\, \Phi(x,y) = \Phi(\lambda x, \lambda y)\,.
\end{equation}
The $p$-adic associator $\Phi$ satisfies the  denominator estimate
\begin{equation}  \label{aux}
\Phi \in \sum_{n=0}^\infty p^{-a_p(n)} \Z_p\langle x,y \rangle^n.
\end{equation}
\end{theorem}

\begin{proof}
Existence and uniqueness of an associator $\Phi \in {\rm Assoc}_1(\mathbb{Q}_p)$ satisfying Equation \eqref{fPhi} follow from Theorem \ref{thm:GT-to-assoc}.

%Since $\lambda$ is not a root of 1, there is a unique element $\Phi\in 1+\k\langle\!\langle x,y \rangle\!\rangle^{\geq1}$ satisfying \eqref{fPhi}. Indeed, the degree $n$ homogeneous part of \eqref{fPhi} is of the form
%\[\Phi_n+Q_n(x,y,\Phi_1,\dots,\Phi_{n-1})=\lambda^n \Phi_n,\]
%where $Q_n$ is a non-commutative polynomial and $\Phi_n$ is the degree $n$ homogeneous part of $\Phi$. We thus have a recurrence relation
%\[\Phi_n=\frac{1}{\lambda^n-1}\,Q_n(x,y,\Phi_1,\dots,\Phi_{n-1}), \]
%which admits a unique solution.
%
%Since $\lambda'=\lambda^{p-1}\in 1+p\Z_p$, the element $(\lambda',f')=(\lambda, f)^{p-1}$ admits a logarithm
%$\ln(\lambda', f') \in \mathfrak{gt}(\mathbb{Q}_p)$. Let
%$(1, \psi)=\ln(\lambda', f')/\ln(\lambda')$. Then $\psi$ and $\Phi$ satisfy relation \eqref{psiphi}. Hence $\Phi\in{\rm Assoc}_1(\mathbb{Q}_p)$.

Let $\hat{\Phi} \in \mathbb{Q}_p\langle\!\langle \hat{x},\hat{y} \rangle\!\rangle$
be the formal power series in $\hat{x}$ and $\hat{y}$, which coincides with
$\Phi \in \mathbb{Q}_p\langle\!\langle x,y \rangle\!\rangle$ under the substitution
$\hat{x}=e^x-1,\ \hat{y}=e^y-1$. We will prove that
$
\hat{\Phi} \in \sum_{n=0}^\infty p^{-a_p(n)} \Z_p\langle
\hat{x},\hat{y} \rangle^n .
$
 By induction, we assume
that $\hat{\Phi}$ has this property up to degree $n-1$. Then the
degree $n$ part of equation \eqref{fPhi} reads
$$
\lambda^n \hat{\Phi}_n(\hat{x},\hat{y}) = \hat{\Phi}_n(\hat{x},\hat{y})
+ F_n(\hat{x}, \hat{y}, \hat{\Phi}_1, \dots, \hat{\Phi}_{n-1}),
$$
where $F_n(\hat{x}, \hat{y}, \hat{\Phi}_1, \dots, \hat{\Phi}_{n-1})$
is a noncommutative polynomial
in $\hat{x}$, $\hat{y}$ and $\hat{\Phi}_k(\hat{x}, \hat{y})$
for $k < n$ with coefficients in $\Z_p$.
Solving for $\hat{\Phi}_n$, we obtain
$$
\hat{\Phi}_n = \frac{1}{\lambda^n -1} \,
F_n(\hat{x}, \hat{y}, \hat{\Phi}_1, \dots, \hat{\Phi}_{n-1}).
$$

For $p>2$, we note that $(\lambda^n-1)^{-1}\in\mathbb{Z}_p$ if $(p-1)\nmid n$ and that $(\lambda^{(p-1)k}-1)^{-1} \in (kp)^{-1}\Z_p$. Since $F_n$ satisfies $F_n(0,0,\hat{\Phi}_1, \dots, \hat{\Phi}_{n-1})=0$, we use \eqref{a1} to obtain
\[F_n(\hat{x}, \hat{y}, \hat{\Phi}_1, \dots, \hat{\Phi}_{n-1}) \in p^{-a_p(n-1)} \Z_p\langle \hat{x}, \hat{y} \rangle^n\]
and thus
\[\hat{\Phi}_n\in
 \begin{cases}
   p^{-a_p(n-1)} \Z_p \langle \hat{x}, \hat{y} \rangle^n &\text{if }(p-1)\nmid n\\
   p^{-a_p(n-1)-1-v_p(n/(p-1))} \Z_p \langle \hat{x}, \hat{y} \rangle^n &\text{if }(p-1)\mid n
 \end{cases}
 \Bigg\}= p^{-a_p(n)} \Z_p \langle \hat{x}, \hat{y} \rangle^n\,.
\]

For $p=2$, we have $(\lambda^n -1)^{-1} \in (4n)^{-1}(1+2\Z_2)$. In this case, the induction argument gives
$$
\hat{\Phi}_n \in n^{-1}  2^{-a_2(n-2)-2} \Z_2 \langle \hat{x}, \hat{y} \rangle^n .
$$
Here we used the fact that, for $(\lambda,f)\in {\rm GT}$, the degree
one component of $f$ always vanishes.
Since
$$a_2(n)  \geq v_2(n)+a_2(n-2)+2\, ,$$ we have
$$\hat{\Phi}_n \in 2^{-a_2(n)} \Z_2 \langle \hat{x}, \hat{y} \rangle^n,$$ as required.
\end{proof}

\begin{remark}
A $p$-adic associator
 $\Phi$ satisfying condition \eqref{aux} is convergent in the $p$\nobreakdash-adic norm as a power series in $x$ and $y$ if $v_p(x),v_p(y)>p/(p-1)^2$; under this condition we have $v_p(\Phi-1)>1/(p-1)$. This follows from the inequality
\[a_p(n)\leq\frac{p}{(p-1)^2}\,n-\frac{1}{p-1}\,.\]
\end{remark}

\begin{remark}
Let $(\lambda, f) \in {\rm GT}_p$ be as in Theorem \ref{Phiestimate},
and let $(0,g)$ be the unique element of the Grothendieck--Teichm\"uller semi-group
defined by the property $(0,g)(\lambda, f)=(0,g)$.
Again, one can show that
$g \in \sum_{n=0}^{\infty} p^{-a_p(n)} \mathbb{Z}_p\langle x,y \rangle^n$.
\end{remark}

%\begin{remark}
%An associator $\Phi$ is called \emph{even} if $\Phi(x,y)=\Phi(-x,-y)$. If we choose $(\lambda,f)\in {\rm GT}_p$ in Theorem \ref{Phiestimate} in such a way that $f(x,y)=f(-x,-y)$ then the corresponding associator $\Phi\in{\rm Assoc}_1(\mathbb{Q}_p)$ is even. Such an element $(\lambda,f)$ exists for every $\lambda$ and every $p$. It can be seen in the following way.
%
%An element $(\lambda,f)\in {\rm GT}_p$ satisfies $f(x,y)=f(-x,-y)$ iff $(\lambda,f)$ commutes with $(-1,1)\in {\rm GT}_p$. Let $c\in\text{Gal}(\bar\Q/\Q)$ denote the complex conjugation. The element $(-1,1)\in {\rm GT}_p$ is the image of $c$ under the group morphism $\text{Gal}(\bar\Q/\Q)\to\text{GT}_p$. It is therefore sufficient to prove that for every $\lambda\in\Z_p^*$ there is an element $d\in\text{Gal}(\bar\Q/\Q)$ such that its $p$-adic cyclotomic character is $\lambda$ and such that $c$ commutes with $d$.
%
%Any element $\lambda\in\Z_p^*\cong\text{Gal}(\Q(\mu_{p^\infty})/\Q)$ preserves $\Q(\mu_{p^\infty})\cap\R$ and therefore it can be lifted to an element $d\in\text{Gal}(\bar\Q/\Q)$ preserving $\bar\Q\cap\R$. This is equivalent to $c\,d=d\,c$.
%\end{remark}

\begin{remark}
In \cite{Furusho_p}, Furusho defined a $p$-adic analogue $\Phi_\textit{KZ}^p$ of the KZ associator,
which turns out to be an element of ${\rm GRT}(\mathbb{C}_p)$ rather than of
${\rm Assoc}_1(\mathbb{Q}_p)$. Hence, one cannot directly apply the considerations of this
paper to this element.
\end{remark}

\subsection{Kashiwara-Vergne factorization of the Jacobson element}
It is natural to ask whether the denominator estimate given by Theorem
\ref{Phiestimate} is optimal. Despite the fact that we are not able to answer
this question in full generality, we will show in this section that
the simple pole (i.e., $p^{-1}$) must appear in degree $p-1$
as suggested by our estimate.

Let $\Phi\in\text{Assoc}_1(\mathbb{Q}_p)$ be such that
$$
\Phi \in 1+\Z_p\langle\!\langle x,y \rangle\!\rangle^{\geq 1} + p^{-1}\Z_p\langle x,y \rangle^{p-1}
+ \mathbb{Q}_p\langle\!\langle x,y \rangle\!\rangle^{\geq p} .
$$
For instance, the $\Phi$'s constructed in Theorem
\ref{Phiestimate} are all of this type.
Let $(1,\psi) \in \mathfrak{gt}(\mathbb{Q}_p)$ be given by relation \eqref{psiphi}. It then satisfies
$$
\psi \in \Z_p\langle\!\langle x,y \rangle\!\rangle^{\geq 1} + p^{-1}\Z_p\langle x,y \rangle^{p-1}
+ \mathbb{Q}_p\langle\!\langle x,y \rangle\!\rangle^{\geq p} .
$$
Denote by $\sigma_p \in \mathfrak{lie}_2(\mathbb{F}_p)$ the modulo $p$ reduction
of the degree $p-1$ part of $p\,\psi$. Recall that
in the free Lie algebra $\mathfrak{lie}_2(\mathbb{F}_p)$ there is a
canonical (Jacobson) element given by the formula $J_p(x,y)=(x+y)^p-x^p-y^p$.

\begin{proposition}  \label{Jacob}
The element $\sigma_p$ satisfies equations \eqref{gt1}, \eqref{grt2},
and \eqref{grt3} and the property
\begin{equation} \label{Jacobson}
[x, \sigma_p(-x-y,x)] + [y, \sigma_p(-x-y,y)]=J_p(x,y),
\end{equation}
\end{proposition}

\begin{proof}
By taking the degree $p-1$ parts of equations \eqref{gt1} and \eqref{gt2},
and extracting the residues (i.e., the coefficients in front of $p^{-1}$),
we obtain  the following equations for $\sigma_p$:
$$
\sigma_p(x,y)=-\sigma_p(y,x) \ , \quad
\sigma_p(x,y)+\sigma_p(y,z)+\sigma_p(z,x)=0
$$
for $x+y+z=0$. Here we used the fact that the degree $p-1$ part of $\ch(x,y)$
has coefficients in $\Z_p$. By applying the same procedure to equation
\eqref{gt3}, we obtain
\begin{multline*}
\sigma_p(x_{1,2}, x_{2,3}+x_{2,4})+\sigma_p(x_{1,3}+x_{2,3},x_{3,4}) = \\
\sigma_p(x_{2,3}, x_{3,4})+\sigma_p(x_{1,2}+x_{1,3},x_{2,4}+x_{3,4})
+\sigma_p(x_{1,2},x_{2,3}),
\end{multline*}
where relations of the Lie algebra $\mathfrak{pb}_4(\mathbb{Q}_p)$ are
replaced by their linearized (graded) version in the Lie algebra
$\mathfrak{t}_4(\mathbb{F}_p)$.

For equation \eqref{Jacobson}, we add up two copies of equation \eqref{gt2},
one for the triple $x$, $y$, $z=-ch(x,y)$ and one for the tripe $y$, $x$, $\tilde{z}=-\ch(y,x)$,
to obtain
$$
(1- \exp(-\ad_x))\psi(z,x) + (\exp(\ad_y) -1 )\psi(z,y) + x+y + \frac{1}{2}(z+\tilde{z})=0.
$$
We consider the residue of the degree $p$ component of this equation, and use the fact
that the residue of the degree $p$ part of $\ch(x,y)$ is
exactly $J_p(x,y)$. This gives \eqref{Jacobson}, as required.
\end{proof}

 Proposition \ref{Jacob} shows that $\sigma_p(x,y)$ does not vanish,
and $\psi$ has to have a simple pole in degree $p-1$. For $\Phi_{p-1}$ the degree $p-1$ homogeneous component of $\Phi$, relation \eqref{psiphi} implies $\sigma_p=-p\,\Phi_{p-1}\text{ mod }p$. Therefore, $\Phi$ has a simple pole in degree $p-1$.

Every element $a \in \k\langle x,y \rangle$ admits a unique decomposition of
the form $a=a_0 + a_1 x + a_2 y$, where $a_0 \in \k$ and $a_1, a_2 \in \k\langle x,y \rangle $.
We will denote $a_1 = \partial_x a, a_2=\partial_y a$. Let $\mathfrak{tr}_2(\k)$ be the quotient
of $\k\langle x,y \rangle$ by the subspace spanned by commutators. The space $\mathfrak{tr}_2(\k)$
is spanned by cyclic words in letters $x$ and $y$. We will denote by
${\rm tr}: \k\langle x,y \rangle \rightarrow \mathfrak{tr}_2(\k)$ the canonical projection.

\begin{proposition}   \label{tr}
Let $a(x,y)=\sigma_p(-x-y,x)$ and $b(x,y)=\sigma_p(-x-y,y)$. Then,
\begin{equation}  \label{KV}
{\rm tr}( x (\partial_x a) + y (\partial_y b)) = \frac{1}{2} {\rm tr}  \left( (x+y)^{p-1}
- x^{p-1} - y^{p-1}\right) .
\end{equation}
\end{proposition}

\begin{proof}
By Propositions 4.2 and 4.3 in \cite{AT}, equations \eqref{gt1}, \eqref{grt2},
and \eqref{grt3} for $\sigma_p(x,y)$ imply that
${\rm tr}( x (\partial_x a) + y (\partial_y b)) = {\rm tr} (f(x+y)-f(x)-f(y))$,
where $f$ is a polynomial in one variable. Since the degree of the left-hand side
is equal to $p-1$, we conclude that $f(x)=\alpha x^{p-1}$ for some $\alpha \in
\mathbb{F}_p$.

In order to determine $\alpha$, let $a(x,y)=u \ad_x^{p-2} y + \dots$,
$b(x,y) = v \ad_x^{p-2} y + \dots$, where the dots stand for the lower degree terms in $x$.
By analyzing the terms of the form $x^{p-1}y$ in equation \eqref{Jacobson},
we find $u=-1$. By looking at the terms of the form $yx^{p-2}y$, we obtain $v=1/2$.
This yields ${\rm tr}(x (\partial_x a) + y (\partial_y b)) = -1/2 \, {\rm tr}(x^{p-2} y) + \dots$.
Hence, $f(x)=x^{p-1}/2$, as required.
\end{proof}

\begin{remark}
Equations \eqref{Jacobson} and \eqref{KV} define the Kashiwara--Vergne type  factorization
problem for the Jacobson element. It was posed by Duflo \cite{Duflo}. A partial solution
was given by the second author in \cite{Masha}. Propositions \ref{Jacob} and \ref{tr}
settle the problem in the positive in full generality.
\end{remark}

\section{Rational associators}

\subsection{The group ${\rm GRT}_\textit{nat}$} Let $p$ be a prime.  Following \cite{Lazard}, we extend the valuation $v_p$ from $\mathbb{Q}$ to $\mathbb{Q}\langle\!\langle x,y\rangle\!\rangle$ by setting $v_p(x)=v_p(y)=p/(p-1)^2$, and consider a subset of $\mathfrak{grt}(\mathbb{Q})$
$$
\mathfrak{grt}_\textit{nat}=\{ \psi \in \mathfrak{grt}(\mathbb{Q});\, v_p(\psi) \geq 1/(p-1)
 \text{ for  all $p$ prime} \} .
$$
In other words, for $\psi(x,y) \in \mathfrak{grt}_\textit{nat}$, we have
$\psi \in \sum_{n=1}^\infty p^{-b(n)} \Z_p\langle x,y \rangle$, where
$$b(n)=\Big[\frac{p}{(p-1)^2}\,n-\frac{1}{p-1}\Big].$$
It is easy to see that $\mathfrak{grt}_\textit{nat}$ is a Lie subalgebra over $\mathbb{Z}$ of
$\mathfrak{grt}(\mathbb{Q})$. Indeed,
$v_p([\phi, \psi]_{\mathfrak{grt}})\geq v_p(\phi)+v_p(\psi) \geq 2/(p-1)$.

Similarly, we introduce
$$
{\rm GRT}_\textit{nat}=\{ g \in {\rm GRT}(\mathbb{Q});\, v_p(g-1) \geq 1/(p-1)
\text{ for  all $p$ prime} \} .
$$
Again, for $g \in {\rm GRT}_\textit{nat}$, we have $g(x,y) \in 1 +
\sum_{n=1}^\infty p^{-b(n)} \Z_p\langle x,y \rangle$. ${\rm GRT}_\textit{nat}$ is
a subgroup of ${\rm GRT}(\mathbb{Q})$.

\begin{proposition}   \label{exp}
The exponential map
$\exp_{\mathfrak{grt}}: \mathfrak{grt}_\textit{nat} \to {\rm GRT}_\textit{nat}$
is a bijection.
\end{proposition}

\begin{proof}
The $\mathfrak{grt}$ exponential map $\exp_{\mathfrak{grt}}: \psi \mapsto g$
is described by the formula $g=\sum_{n=0}^\infty g_n/n!$, where
$g_0=1$ and $g_n=\psi \cdot g_{n-1} + g_{n-1} \psi$. In particular, $g_1=\psi$
and $v_p(g_1)\geq 1/(p-1)$. By induction, we obtain $v_p(g_n) \geq n/(p-1)$,
and so $v_p(g_n/n!)=v_p(g_n)-v_p(n!) \geq 1/(p-1)$, where we used the estimate
$v_p(n!) \leq (n-1)/(p-1)$. Hence, $v_p(g-1) \geq 1/(p-1)$, as required.

The ${\rm GRT}$ logarithmic map $\ln_{\text{GRT}}: g \mapsto \psi$ is given by
$\psi= \sum_{n=1}^\infty (-1)^{n+1} \psi_n/n$, where $\psi_1=g-1$
and
$$\psi_n(x,y)=\psi_{n-1}(gxg^{-1},y)-\psi_{n-1}(x,y)+\psi_{n-1}(gxg^{-1},y)\,\psi_1(x,y).$$
Using $v_p(\psi_1)=v_p(\g-1)\geq 1/(p-1)$, we prove by induction
that $v_p(\psi_n)\geq n/(p-1)$. Therefore, $v_p(\psi_n/n) \geq n/(p-1)-v_p(n)
\geq n/(p-1) - v_p(n!) \geq 1/(p-1)$, and $v_p(\psi) \geq 1/(p-1)$ as required.
\end{proof}

%\begin{remark}
%Note that we could take $v_p(x)=v_p(y)=u$
%and $v_p(\psi) \geq v,$ $v_p(g-1) \geq v$ with an arbitrary $u$ and with $v \geq 1/(p-1)$ in the definitions of
%$\mathfrak{grt}_\textit{small}$ and ${\rm GRT}_\textit{small}$.  In this case,  Proposition
%\ref{exp} remains valid. In particular, for $p>2$, we take $u=1/p+1/(p-1)$
%and $v=1/(p-1)$, and, for $p=2$, we take $u=2$ and $v=1$. We denote the corresponding
%Lie algebra by $\mathfrak{grt}_\textit{big}$.
%The elements of $\mathfrak{grt}_\textit{big}$ are of the form
%$\psi \in \sum_{n=1}^\infty p^{-c_p(n)} \Z_p\langle x,y \rangle$,
%where $c_p(n)=[ n/p+(n-1)/(p-1)]$ for $p>2$ and $c_2(n)=2n-1$.
%The corresponding group is denoted by ${\rm GRT}_\textit{big}$,
%and its elements are of the form $g(x,y) \in 1 +
%\sum_{n=1}^\infty p^{-c_p(n)} \Z_p\langle x,y \rangle$.
%\end{remark}

%\begin{proposition}  \label{preserves}
% The group ${\rm GRT}_\textit{small}$ preserves the set ${\rm Assoc}_\textit{nat}$
% under the natural action.
% \end{proposition}
%
% \begin{proof}
% This follows from the inequality $b_p(n) \leq a_p(n)$.
% \end{proof}
\begin{remark}
The inequality $v_p(k!)\leq (k-1)/(p-1)$ implies
\[a_p(n)\leq b_p(n).\]
\end{remark}

\subsection{Natural rational associators}
We call a rational associator $\Phi\in\mathrm{Assoc}_1(\mathbb{Q})$
{\em natural} if it satisfies the denominator estimates
$\Phi \in 1+\sum_{n=1}^\infty \, p^{-b_p(n)} \Z_p\langle x,y\rangle^n$ for all
prime $p$, i.e., if $v_p(\Phi-1)\geq1/(p-1)$ for every $p$. We denote by ${\rm Assoc}_\textit{nat}$ the set of (rational) natural
associators; similarly, we set $\text{Assoc}_\textit{nat}^{(k)}=\mathrm{Assoc}_1(\mathbb{Q})^{(k)}\cap \big(1+\sum_{n=1}^k \, p^{-b_p(n)} \Z_p\langle x,y\rangle^n\big)$. The main result of this paper is as follows:

\begin{theorem}   \label{maintheorem}
The set of natural associators ${\rm Assoc}_\textit{nat}$ is nonempty, and, for every $k$, the truncation map ${\rm Assoc}_\textit{nat}\to{\rm Assoc}_\textit{nat}^{(k)}$ is surjective.
\end{theorem}

\begin{proof}
For each prime $p$, we choose an associator $\Phi_p\in{\rm Assoc}_1(\mathbb{Q}_p)$ satisfying denominator estimates \eqref{aux}.
Let $(1,\psi_p)$ correspond to $\Phi_p$ via relation \eqref{psiphi}.
We denote by $(1,\psi_p^{(k)})$ and $\Phi_p^{(k)}$ their degree $k$ truncations.

We choose an element $(1, \tilde{\psi}) \in \mathfrak{gt}(\mathbb{Q})$.
For each $k \in \mathbb{N}$, let $(1, \tilde{\psi}^{(k)}) \in \mathfrak{gt}^{(k)}(\mathbb{Q})$
be its degree $k$ truncation, and let $\phi_i, i=1, \dots, N_k$ be
an integral basis in
$\mathfrak{gt}_1^{(k)}(\mathbb{Q})$ (i.e., $\phi_i$'s are non-commutative polynomials in $x$ and $y$
with integer coefficients). Let $P_k$ be the finite set of primes that either do not exceed $k$
or enter denominators of $\tilde{\psi}^{(k)}$.
For each $p \in P_k$, we choose two sufficiently large
positive integers $s_k(p), t_k(p)$ (the choice will be discussed later).

For each $p\in P_k$, we have
$$
(1, \psi_p^{(k)})=(1, \tilde{\psi}^{(k)}) + \sum_{i=1}^{N_k} \kappa_{p,i} \phi_i
$$
with $\kappa_{p,i} \in \mathbb{Q}_p$. Put $M_k=\prod_{p\in P_k} p^{s_k(p)}$. If the numbers
$s_k(p)$ are sufficiently large, we have $M\kappa_{p,i} \in \mathbb{Z}_p$ for all
$p \in P_k$ and for all $i$. By Chinese Remainder Theorem, we can choose
integers $\mu_i \in \mathbb{Z}$ such that $\mu_i - M\kappa_{p,i} =0 \, {\rm mod} \, p^{t_k(p)}$
for all $p \in P_k$. Define an element
$$
(1, \psi^{(k)})=(1, \tilde{\psi}^{(k)}) + \frac{1}{M_k} \, \sum_{i=1}^{N_k} \mu_i \phi_i
\in \mathfrak{gt}^{(k)}(\Q) ,
$$
and the corresponding (by Proposition \ref{uptodegreek}) associator
$\Phi^{(k)} \in {\rm Assoc}_1^{(k)}(\mathbb{Q})$.

Note that, for $p \notin P_k$, the coefficients of $\psi^{(k)}$ are in $\mathbb{Z}_p$.
Since $p>k$, the associator $\Phi^{(k)}$ also has coefficients in $\mathbb{Z}_p$.
Let $p\in P_k$. Since the map $\psi^{(k)} \mapsto \Phi^{(k)}$ (defined by equation
\eqref{psiphi}) is continuous in the $p$-adic topology, we can choose sufficiently large $t_k(p)$'s
so that $\Phi^{(k)}$ be arbitrarily close to $\Phi^{(k)}_p$. In particular,
we may assume that $\Phi^{(k)} \in {\rm Assoc}_\textit{nat}^{(k)}$.

Denote by $\pi_k: {\rm Assoc}_1^{(k+1)}(\mathbb{Q}) \to {\rm Assoc}_1^{(k)}(\mathbb{Q})$
the natural projection
forgetting the degree $k+1$ part of an associator. Consider the associators
$\Phi^{(k)}$ and $\pi_k(\Phi^{(k+1)}) \in {\rm Assoc}_1^{(k)}(\mathbb{Q})$. There is a unique
element $g_k \in {\rm GRT}^{(k)}(\mathbb{Q})$ taking one of them to the other.
Let $p \notin P_k \cap P_{k+1}$. Then both  $\Phi^{(k)}$ and  $\pi_k(\Phi^{(k+1)})$
have coefficients in $\mathbb{Z}_p$, and so does $g_k$. For $p \in P_k \cap P_{k+1}$,
we choose sufficiently large $t_k(p)$ and $t_{k+1}(p)$ to obtain
$g_k \in {\rm GRT}_\textit{nat}^{(k)}$. By Proposition \ref{exp},
there is an isomorphism between $\mathfrak{\grt}_\textit{nat}$ and ${\rm GRT}_\textit{nat}$.
Since all elements of $\mathfrak{\grt}_\textit{nat}^{(k)}$ lift to elements
of $\mathfrak{\grt}_\textit{nat}$, the same applies to lifting from
${\rm GRT}_\textit{nat}^{(k)}$ to ${\rm GRT}_\textit{nat}$.
Hence, $g_k$ can be lifted
(in many ways) to an element $\hat{g}_k \in {\rm GRT}_\textit{nat}$.

We consider the sequence of elements
$$
\Phi(k)=(\hat{g}_k \hat{g}_{k-1} \dots \hat{g}_2) \cdot \Phi^{(k)} \in
{\rm Assoc}_1^{(k)}(\mathbb{Q}),
$$
where the action of ${\rm GRT}$ on ${\rm Assoc}_1^{(k)}(\mathbb{Q})$ is defined by forgetting
degrees higher than $k$. Since ${\rm GRT}_\textit{nat}$ and ${\rm Assoc}_\textit{nat}$ are defined by the same denominator conditions,
  the action of the group
${\rm GRT}_\textit{nat}$ preserves ${\rm Assoc}_\textit{nat}$ and is free and transitive. We thus have $\Phi(k) \in {\rm Assoc}_\textit{nat}^{(k)}$. By construction of $\hat{g}_k$, we have
$\pi_k(\Phi(k+1))=\Phi(k)$. Hence, the sequence $\Phi(k)$ defines an element of
${\rm Assoc}_\textit{nat}$. That is, the set ${\rm Assoc}_\textit{nat}$ is nonempty, as required.

Finally, since the action of ${\rm GRT}_\textit{nat}$ on ${\rm Assoc}_\textit{nat}$ is free and transitive, any choice of an associator $\Phi\in{\rm Assoc}_\textit{nat}$ gives a bijection ${\rm GRT}_\textit{nat}\to{\rm Assoc}_\textit{nat}$, which is moreover
compatible with the truncation maps. The surjectivity of ${\rm Assoc}_\textit{nat}\to{\rm Assoc}_\textit{nat}^{(k)}$ therefore
follows from surjectivity of ${\rm GRT}_\textit{nat}\to{\rm GRT}_\textit{nat}^{(k)}$.
\end{proof}

\begin{remark}
Let ${\rm Assoc}'_\textit{nat}\subset{\rm Assoc}_\textit{nat}$ be the set of rational associators satisfying the
denominator estimate $\Phi \in \sum_{n=0}^\infty p^{-a_p(n)} \Z_p\langle x,y \rangle^{n}$.
We can still prove that ${\rm Assoc}'_\textit{nat}$ is nonempty, if we replace ${\rm GRT}_\textit{nat}$
with a similar group given by $v_p(x)=v_p(y)=1/(p-1)$. However, we were not able
to establish the surjectivity property for truncation maps of ${\rm Assoc}'_\textit{nat}$.
\end{remark}

\begin{remark}
Theorem \ref{maintheorem} remains valid if we require natural associators
to be even (that is, $\Phi(x,y)=\Phi(-x,-y)$). Indeed, let
$\Phi^{(2k)}\in\text{Assoc}_1^{(2k)}(\Q)$ be an even natural associator up
to degree $2k$. By Remark 1 (after Proposition 5.8) in \cite{Dr},
$\Phi^{(2k)}$ can be extended by zero in degree $2k+1$ to obtain an even
natural associator up to degree $2k+1$. Then, by theorem
\ref{maintheorem}, $\Phi^{(2k)}$ extends (in many ways) to a natural
associator $\Phi^{(2k+2)}$ up to degree $2k+2$. $\Phi^{(2k+2)}$ is even
since $\Phi^{(2k)}$ is even and the contribution in degree $2k+1$
vanishes. Hence, even natural associators can be extended from lower to
higher degrees without obstructions, as required.
\end{remark}

\begin{remark}
In \cite{AET}, it is proved that the Knizhnik--Zamolodchikov associator $\Phi_\textit{KZ}$
has a finite convergence radius (as a non-commutative power series). It is natural to
conjecture the existence of rational natural associators with finite convergence
radius. We do not know how to prove (or disprove) this conjecture. More generally, it would be interesting to establish the existence of rational associators with given bounds on both denominators and numerators.
\end{remark}

\section{Kontsevich knot invariant}

Drinfeld associators can be used to define a universal finite type
invariant of knots in $\R^3$. Below we give a brief account of this
construction.

\begin{figure}[h]
\begin{tikzpicture}
\draw[thick,->] (0,0) circle (1cm);
\draw[thick,->] (1,0)--(1,0.05);
\draw[thick,->] (-1,0)--(-1,-0.05);
\draw[thin,red] (30:1cm) -- (140:1cm) (80:1cm)--(200:1cm) (290:1cm)--(160:1cm) (240:1cm)--(320:1cm);
\end{tikzpicture}
\caption{A chord diagram}\label{fig:chord}
\end{figure}
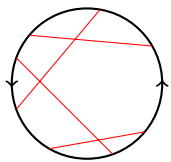

Let $\k$ be a field of characteristic zero. A chord diagram (see Fig.~\ref{fig:chord})
is an oriented circle with a choice of pairs of marked points modulo
orientation preserving
diffeomorphisms. The two points of a pair are
connected
by a chord.  The algebra ${\rm Chord}(\k)$ is spanned by the
chord diagrams modulo the 4T relation
\[
\begin{tikzpicture}[scale=0.7,baseline=-0.5ex]
\draw[thick,->] (-20:1cm) arc (-20:20:1cm);
\draw[thick,->] (100:1cm) arc (100:140:1cm);
\draw[thick,->] (220:1cm) arc (220:260:1cm);
\draw[dotted] (0,0) circle (1cm);
\draw[thin,red] (0:1cm)--(115:1cm) (240:1cm)--(125:1cm);
\end{tikzpicture}
\;-\;
\begin{tikzpicture}[scale=0.7,baseline=-0.5ex]
\draw[thick,->] (-20:1cm) arc (-20:20:1cm);
\draw[thick,->] (100:1cm) arc (100:140:1cm);
\draw[thick,->] (220:1cm) arc (220:260:1cm);
\draw[dotted] (0,0) circle (1cm);
\draw[thin,red] (0:1cm)--(130:1cm) (240:1cm)--(110:1cm);
\end{tikzpicture}
\;+\;
\begin{tikzpicture}[scale=0.7,baseline=-0.5ex]
\draw[thick,->] (-20:1cm) arc (-20:20:1cm);
\draw[thick,->] (100:1cm) arc (100:140:1cm);
\draw[thick,->] (220:1cm) arc (220:260:1cm);
\draw[dotted] (0,0) circle (1cm);
\draw[thin,red] (0:1cm)--(245:1cm) (235:1cm)--(120:1cm);
\end{tikzpicture}
\;-\;
\begin{tikzpicture}[scale=0.7,baseline=-0.5ex]
\draw[thick,->] (-20:1cm) arc (-20:20:1cm);
\draw[thick,->] (100:1cm) arc (100:140:1cm);
\draw[thick,->] (220:1cm) arc (220:260:1cm);
\draw[dotted] (0,0) circle (1cm);
\draw[thin,red] (0:1cm)--(230:1cm) (250:1cm)--(120:1cm);
\end{tikzpicture}
\;=\;0
\]
and modulo
the relation
\[
\begin{tikzpicture}[scale=0.7,baseline=0.3cm]
\draw[thick,->] (-1,1) parabola bend(0,0) (1,1);
\draw[dotted] (-1.2,1.44) parabola bend(0,0) (1.2,1.44);
\draw[thin,red] (-0.8,0.64)--(0.8,0.64);
\end{tikzpicture}
=\;0.
\]
 The product in ${\rm
Chord}(\k)$
is defined as a connected sum of chord diagrams.
The
4T relation insures that the product map is well-defined.
\[
\begin{tikzpicture}[scale=0.7,baseline=-0.5ex]
\draw[thick](0,0) circle (1cm);
\draw[thin, red] (55:1cm)--(205:1cm) (305:1cm)--(155:1cm) (90:1cm)--(270:1cm);
\draw[thick,->](1,0)--(1,0.05);
\end{tikzpicture}
\;\times\;
\begin{tikzpicture}[scale=0.7,baseline=-0.5ex]
\draw[thick](0,0) circle (1cm);
\draw[thin, red] (60:1cm)--(240:1cm) (-60:1cm)--(-240:1cm);
\draw[thick,->](1,0)--(1,0.05);
\end{tikzpicture}
\;=
\begin{tikzpicture}[scale=0.7,baseline=-0.5ex]
\draw[thick] (30:1cm) arc (30:330:1cm);
\draw[thick] (2.5,0) +(-150:1cm) arc (-150:150:1cm);
\draw[thick] (30:1cm)..controls ++(-60:0.5) and ++(-120:0.5) .. (1.634,0.5) (-30:1cm)..controls ++(60:0.5) and ++(120:0.5) .. (1.634,-0.5);
\draw[thick,->] (-30:1cm)--+(60:0.05cm);
\draw[thick,->,xshift=2.5cm] (150:1cm)--+(-120:0.05cm);
\draw[thin, red] (55:1cm)--(205:1cm) (305:1cm)--(155:1cm) (90:1cm)--(270:1cm);
\draw[thin, red, xshift=2.5cm] (60:1cm)--(240:1cm) (-60:1cm)--(-240:1cm);

\end{tikzpicture}
\]

The algebra ${\rm Chord}(\k)$ is graded with the grading given by the number of chords.
We assume that ${\rm Chord}(\k)$ is degree completed to allow
infinite
linear combinations of chord diagrams.
The linear combinations of chord diagrams with integer coefficients span a
lattice
${\rm Chord}(\Z)\in {\rm Chord}(\k)$.
We define a subring (over $\Z$) of ${\rm Chord}(\mathbb{Q})$,
$$
{\rm Chord}_{nat}=\sum_{n=0}^\infty \, D(n)^{-1} {\rm Chord}_n(\Z),
$$
where $D(n)$ is given by equation \eqref{D(n)} and
${\rm Chord}_n(\Z)$ is the graded component of degree $n$ of ${\rm
Chord}(\Z)$.

There are several steps in constructing the Kontsevich knot invariant
(for details, see \cite{Bar-Natan 2,Le Murakami}).
First, one associates to the knot $K$ a knot diagram $\tilde{K}$
(a projection of $K$ onto a 2-plane). By a small perturbation of the embedding,
one can choose $\tilde{K}$ in such a way that it has only transversal
simple crossings. One can also assume that the height (the $y$-coordinate) is a Morse function and that all the crossings and critical points of $y$ have different heights.

One can then decompose the knot
diagram into a sequence of basic elements. There are six possible basic elements
(up to change of orientation),
\[
a=\tikz{\draw[thick,->] (0,0) arc (180:0:0.3cm);}
\;,\quad
a'=\tikz{\draw[thick,->] (0,0) arc (0:-180:0.3cm);}
,\quad
b=\tikz[scale=0.3, baseline=-0.5ex]{\draw[thick,->](1,-1)--(-1,1);
\draw[line width=4pt, white] (-1,-1)--(1,1);
\draw[thick,->] (-1,-1)--(1,1);}
,\quad
b^{-1}=\tikz[scale=0.3, baseline=-0.5ex]{\draw[thick,->](-1,-1)--(1,1);
\draw[line width=4pt, white] (1,-1)--(-1,1);
\draw[thick,->] (1,-1)--(-1,1);}
,\quad
c=\tikz[scale=0.3, baseline=-0.5ex]{\draw[thick,->](-1,-1)--(-1,1);
\draw[thick,->] (1,-1)--(1,1);
\draw[thick,->](-0.4,-1)..controls ++(0,1) and ++(0,-1) .. (0.4,1);}
\;,\quad
c^{-1}=\tikz[scale=0.3, baseline=-0.5ex]{\draw[thick,->](-1,-1)--(-1,1);
\draw[thick,->] (1,-1)--(1,1);
\draw[thick,->](0.4,-1)..controls ++(0,1) and ++(0,-1) .. (-0.4,1);}\;.
\]

Then one decorates the basic elements with (horizontal) chords according
to the
following rules: elements $a$ and $a'$ remain unchanged, the element $b$
is decorated
with $\exp(t/2)$, where $t$ is a horizontal chord extending between two
strands,
\[
\tikz[scale=0.3,baseline=-0.5ex]{\draw[thick,->](1,-1)--(-1,1);
\draw[line width=4pt, white] (-1,-1)--(1,1);
\draw[thick,->] (-1,-1)--(1,1);}
\longrightarrow
\tikz[scale=0.3,baseline=-0.5ex]{\draw[thick,->](1,-1)--(-1,1);
\draw[line width=4pt, white] (-1,-1)--(1,1);
\draw[thick,->] (-1,-1)--(1,1);}
+\frac{1}{2}
\tikz[scale=0.3,baseline=-0.5ex]{\draw[thick,->](1,-1)--(-1,1);
\draw[line width=4pt, white] (-1,-1)--(1,1);
\draw[thick,->] (-1,-1)--(1,1);
\draw[thin, red] (-0.5,0.5)--(0.5,0.5);}
+\frac{1}{2^2\times 2!}
\tikz[scale=0.3,baseline=-0.5ex]{\draw[thick,->](1,-1)--(-1,1);
\draw[line width=4pt, white] (-1,-1)--(1,1);
\draw[thick,->] (-1,-1)--(1,1);
\draw[thin, red] (-0.4,0.4)--(0.4,0.4) (-0.6,0.6)--(0.6,0.6);}
+\cdots\;.
\]
 Similarly, the element $b^{-1}$ is decorated with
$\exp(-t/2)$. Finally, one chooses a Drinfeld associator $\Phi\in\text{Assoc}_1(\k)$ and
decorates
elements $c$ and $c^{-1}$ with
$\Phi(t_{1,2}, t_{2,3})$ and $\Phi(t_{1,2}, t_{2,3})^{-1}$, respectively.
Here $t_{1,2}$ is a horizontal chord extending between strands $1$ and $2$
and $t_{2,3}$ is a horizontal chord extending between strands $2$ and $3$. To get a decoration of basic elements with some
strands with reversed orientation one uses the following rule: let $S$ be
a strand with $k$ marked points (ends of various chords). Then reversing
orientation of $S$ is accompanied by adding the sign $(-1)^k$.

The decorated knot diagram is an element of ${\rm Chord}(\k)$. We denote it by $I_\Phi(\tilde{K})$. By abuse of notation, we denote by $\tilde{K}$
the knot diagram together with its decomposition into basic elements.

\begin{proposition}
Let $\Phi$ be a natural rational associator. Then
$I_\Phi(\tilde{K}) \in {\rm Chord}_{nat}$, for all
knot diagrams $\tilde{K}$.
\end{proposition}

\begin{proof}
Assume that the knot diagram $\tilde{K}$ decomposes into $s$ basic
elements
of type $b$ and $b^{-1}$, $t$ basic elements of type $c$ and $c^{-1}$, and some number of basic
elements of type $a$ and $a'$. Then denominators of degree $n$ contributions
in $I_\Phi(\tilde{K})$ are of the form
$$
D(k_1) \dots D(k_s) (2^{l_1} l_1!) \dots (2^{l_t} l_t!),
$$
where $k_1 + \dots k_s + l_1 + \dots + l_t=n$. Here contributions
$D(k_1), \dots, D(k_s)$ come from denominators of the natural
associator $\Phi$, and remaining terms are denominators
of the exponential functions $\exp(t/2)$ and $\exp(-t/2)$.

Since $b_p(k+l) \geq b_p(k) + b_p(l)$, the product $D(k_1) \dots D(k_s)$ is a
divisor of $D(k)$ for $k=k_1+\dots +k_s$. Furthermore,
for $p\geq 3$, we have  $b_p(k+l) \geq b_p(k) + v_p(l!)$
and $b_2(k+l) \geq b_2(k) + v_2(l!)+l$. Hence,
$D(k) (2^{l_1} l_1!) \dots (2^{l_t} l_t!)$ is a divisor
$D(n)=D(k+l)$ for $l=l_1+\dots + l_t$.
\end{proof}

Note that, for all $\tilde{K}$, we have
$I_\Phi(\tilde{K}) \in 1 + {\rm Chord}_{\geq 1}(\k)$. Hence, if $\Phi$
is a natural rational associator, $I_\Phi(\tilde{K})$ is an invertible
element
of ${\rm Chord}_{nat}$.

Let $\tilde{K}_0$ be the following knot diagram:
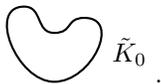
\[
\tikz[scale=0.5]{\draw[thick] (-0.8,1)..controls ++(0.5,0) and ++(-0.5,0)..(0,0.3) ..controls ++(0.5,0) and ++(-0.5,0)..(0.8,0.7)
 ..controls ++(1,0) and ++(1,0)..(0,-1)
 ..controls ++(-1,0) and ++(-1,0)..(-0.8,1) (2,-0.2) node{$\tilde{K}_0$};}.
\]
The Kontsevich invariant of $K$ is defined as 
$$
I(K) = I_\Phi(\tilde{K}) I_\Phi(\tilde{K}_0)^{-X},
$$
where $X$ is the number of local maxima of the height function on
$\tilde{K}$.
The information about $I(K)$ is summarized in the following theorem (see
\cite{Le Murakami}).

\begin{theorem}
$I(K)$ is a knot invariant. That is, it is independent of the choice of
the knot diagram associated to the knot $K$, of the decomposition of
$\tilde{K}$ into
basic elements, and of the choice of a Drinfeld associator $\Phi$.
\end{theorem}

The main result of this Section is the following theorem.

\begin{theorem}\label{something}
The Kontsevich invariant $I(K) \in {\rm Chord}_{nat}$, for all knots $K$.
\end{theorem}

\begin{proof}
Let $K$ be a knot and $\Phi$ be a natural rational associator.
For any knot diagram $\tilde{K}$ of $K$ and for any decomposition
of $\tilde{K}$ into basic elements, we have $I_\Phi(\tilde{K}) \in {\rm
Chord}_{nat}$.
We also have $I_\Phi(\tilde{K}_0) \in {\rm Chord}_{nat}$, and
$I_\Phi(\tilde{K_0})$ is an invertible element. Hence,
$I(K) \in {\rm Chord}_{nat}$, as required.
\end{proof}

\begin{remark}
A weaker estimate (the function $\tilde{b}_p(n)$ is quadratic in $n$)
on denominators in the Kontsevich knot invariant was obtained
in \cite{Thang Le}. To get this estimate, the author uses
associators with values in ``Chinese characters'' instead
of Drinfeld associators.

\end{remark}
\appendix

\section{ Denominator estimates for elements of $\mathfrak{gt}(\mathbb{Q}_p)$}
Here we prove an estimate on the elements $(1,\psi)\in\gt(\mathbb{Q}_p)$ corresponding to the $p$-adic associators constructed in Theorem \ref{Phiestimate} via relation \eqref{psiphi}. In contrast to the linear estimate of Theorem \ref{Phiestimate}, the new estimate has a logarithmic growth. However, this estimate is only satisfied by $\psi$ viewed as a noncommutative series in $\hat{x}=e^x-1$ and $\hat{y}=e^y-1$.

For $p>2$, let $(\lambda,f)\in\text{GT}(\mathbb{Q}_p)$ and $\Phi\in {\rm Assoc}_1(\mathbb{Q}_p)$ be as in Theorem \ref{Phiestimate}. Define $(1,\psi)=\ln(\lambda',f')/\ln\lambda'\in\gt(\mathbb{Q}_p)$, where $(\lambda',f')=(\lambda,f)^{p-1}$.
Then $\Phi$ and $\psi$ satisfy relation \eqref{psiphi}. Since $\Phi^{(p-2)}$ has coefficients in $\Z_p$, the same is true for $\psi^{(p-2)}$. Using the facts that $\ln\lambda'\in p\Z_p$, $f'\in\Z_p\langle\!\langle \hat{x},\hat{y} \rangle\!\rangle$, and $(\lambda',f')=\exp \big(\ln(\lambda')(1,\psi)\big)$, we obtain
\begin{equation}\label{f'est}
f'\in 1+p\Z_p\langle\!\langle \hat{x},\hat{y} \rangle\!\rangle^{\geq1}
  +\Z_p\langle\!\langle \hat{x},\hat{y} \rangle\!\rangle^{\geq p-1}.
\end{equation}

\begin{proposition}  \label{psi}
Let $p>2$ be a prime, and  let $(\lambda', f') \in {\rm GT}_p$ be such that
$\lambda' \in 1+p\Z_p^*$ and $f'$ satisfies \eqref{f'est}. Then the element $(1,\psi)=\ln(\lambda', f')/\ln(\lambda')
\in \mathfrak{gt}(\mathbb{Q}_p)$ is of the form
\begin{equation}  \label{p>2}
\psi \in \Z_p\langle\!\langle \hat{x},\hat{y} \rangle\!\rangle +
\sum_{s=0}^\infty p^{-(s+1)} \Z_p\langle\!\langle \hat{x},\hat{y}
\rangle\!\rangle^{\geq p^s(p-1)} \, \, .
\end{equation}
\end{proposition}

\begin{proof}
Using the Taylor series expansion of $\ln(\lambda', f')$, we obtain
$$
\psi= \, \frac{1}{\ln(\lambda')} \, \sum_{n=1}^\infty \,
\frac{(-1)^{n+1}}{n} \, \psi_n,
$$
where $\psi_1(x,y)=f'(x,y)-1$, and
\begin{align*}
\psi_n(x,y) & =
\psi_{n-1}(\lambda f'xf'^{-1}, \lambda y)\,f(x,y) - \psi_{n-1}(x,y) \\
& =
\psi_{n-1}(\lambda f'xf'^{-1}, \lambda y)\,\psi_1(x,y) +
\big(\psi_{n-1}(\lambda f'xf^{-1}, \lambda y)-\psi_{n-1}(x,y)\big).
\end{align*}
Note that the transformation $x \mapsto \lambda x$ corresponds to
$$
\hat{x} \mapsto (1+\hat{x})^\lambda -1 \equiv
(1+p) \hat{x} \,\, {\rm mod}\, (p^2\hat{x}, p\hat{x}^2, \hat{x}^p),
$$
and that $\psi_1 \in p\Z_p \langle\!\langle \hat{x}, \hat{y} \rangle\!\rangle
+ \Z_p \langle\!\langle \hat{x}, \hat{y} \rangle\!\rangle^{\geq p-1}$.
By induction on $n$, we verify that
$$
\psi_n \in \sum_{k=0}^n p^{k}
\mathbb{Z}_p\langle\!\langle \hat{x}, \hat{y} \rangle\!\rangle^{\geq(n-k)(p-1)} .
$$
Thus, we have
$$
\psi \in \sum_{n=1}^\infty \sum_{k=0}^n n^{-1} p^{k-1}
\mathbb{Z}_p\langle\!\langle \hat{x}, \hat{y} \rangle\!\rangle^{\geq (n-k)(p-1)} .
$$
Let $s+1=1-k+v_p(n)$ be the exponent of $p$ in the denominator.
For $s\geq 0$, we have $v_p(n)=k+s$ and $n \geq p^{k+s}$. Then
$n-k$ takes its minimal value $p^{s}$ at $k=0$.
This yields
$$
\psi \in \Z_p\langle\!\langle \hat{x},\hat{y} \rangle\!\rangle +
\sum_{s=0}^\infty p^{-(s+1)}
\mathbb{Z}_p\langle\!\langle \hat{x}, \hat{y} \rangle\!\rangle^{\geq p^s(p-1)},
$$
as required.
\end{proof}

\begin{remark}
When viewed as a power series in $\hat{x}$ and $\hat{y}$, the element
$\psi$ in Proposition \ref{psi} is convergent in the $p$-adic topology provided
$p$-adic valuations $v_p(\hat{x})$, and $v_p(\hat{y})$ are strictly
positive. As a power series in $x$ and $y$, it converges
provided $\hat{x}$ and $\hat{y}$ are convergent power series, i.e., if $v_p(x),v_p(y) > 1/(p-1)$.
\end{remark}

In the case of $p=2$ we get the following result.

\begin{proposition}  \label{psi2}
Let $(\lambda, f) \in {\rm GT}_2$ be such that $\lambda \in 1 + 4 + 8 \Z_2$.
Then the element $(1, \psi)=\ln(\lambda,f)/\ln(\lambda) \in \mathfrak{gt}(\mathbb{Q}_2)$
is of the form
\begin{equation} \label{p=2}
\psi \in \sum_{s=0}^\infty  2^{-(s+2)} \mathbb{Z}_2\langle\!\langle
\hat{x}, \hat{y} \rangle\!\rangle^{\geq 2^{s}}.
\end{equation}
\end{proposition}

\begin{proof}
The degree one contribution in $f$ vanishes. Hence,
for  $(\lambda, f) \in {\rm GT}_2$, we have  $f \in 1 +
\Z_2\langle\!\langle \hat{x},\hat{y} \rangle\!\rangle^{\geq 2}$. That is,
$\psi_1 = f-1 \in \Z_2\langle\!\langle \hat{x},\hat{y} \rangle\!\rangle^{\geq 2}$.
Similarly to the previous proposition,
we obtain
$$
\psi_n \in \sum_{k=0}^n 2^k \mathbb{Z}_2\langle\!\langle
\hat{x}, \hat{y} \rangle\!\rangle^{\geq n-k} ,
$$
and
$$
\psi \in \sum_{n=1}^\infty \sum_{k=0}^n n^{-1} 2^{k-2} \mathbb{Z}_2\langle\!\langle
\hat{x}, \hat{y} \rangle\!\rangle^{\geq n-k} .
$$
Here we used the fact that $\ln(\lambda) \in 4\mathbb{Z}_2^*$.
By putting $s+2=v_2(n) + 2-k$, we get $v_2(n)=k+s$, and, for $s\geq 0$,
we obtain the estimate $n \geq 2^{k+s}$. Again, the minimum of $n-k$ is attained
at $k=0$, which implies the desired estimate for $\psi$,
$$
\psi \in \sum_{s=0}^\infty  2^{-(s+2)} \mathbb{Z}_2\langle\!\langle
\hat{x}, \hat{y} \rangle\!\rangle^{\geq 2^{s}} .
$$
Note that the study of $s=-1$ (corresponding to the first order pole)
is not needed since we already have at least $1/2^2$ in all degrees.
\end{proof}


\begin{thebibliography}{99}

\bibitem{AET}
A. Alekseev, B. Enriquez, C. Torossian,
Drinfeld associators, braid groups and explicit solutions of
the Kashiwara-Vergne equations, preprint arXiv:0903.4067.


\bibitem{AT}
A. Alekseev, C. Torossian,
The Kashiwara--Vergne conjecture and Drinfeld's associators,
preprint arXiv:0802.4300.


\bibitem{AT'}
A. Alekseev, C. Torossian,
Kontsevich deformation quantization and flat connections,
preprint arXiv:0906.0187

\bibitem{Bar-Natan 2}
D. Bar-Natan, Non-associative tangles.  Geometric topology (Athens, GA, 1993),  139--183, AMS/IP Stud. Adv. Math., 2.1, Amer. Math. Soc., Providence, RI, 1997.

\bibitem{Bar-Natan}
D. Bar-Natan, On associators and the Grothendieck--Teichmuller group. I.
Selecta Math. (N.S.)  {\bf  4},  no. 2, 183--212 (1998).

\bibitem{Dr}
V.G. Drinfeld, On quasi-triangular quasi-Hopf algebras and a group
closely connected with ${\rm Gal}(\overline{\mathbb{Q}}/\mathbb{Q})$,
Leningrad Math. J., vol. \textbf{2} no. 4, 829--860, (1991).

\bibitem{Duflo}
M. Duflo, private communications.

\bibitem{EK}
P. Etingof, D. Kazhdan, Quantization of Lie bialgebras. I.  Selecta Math. (N.S.)
{\bf 2},  no. 1, 1--41, (1996).

\bibitem{Furusho}
H. Furusho,
Pentagon and hexagon equations, preprint arXiv:math/0702128.

\bibitem{Furusho_p}
H. Furusho,
$p$-adic multiple zeta values. I. $p$-adic multiple polylogarithms and the $p$-adic KZ equation,
Invent. Math.  {\bf 155} ,  no. 2, 253--286, (2004).

\bibitem{numbertheory}
H. Furusho,
Double shuffle relation for associators,
preprint arXiv:0808.0319.

\bibitem{KV}
M. Kashiwara and M. Vergne, The Campbell-Hausdorff formula and invariant hyperfunctions.  Invent. Math.  47  (1978), no. 3, 249--272.

\bibitem{Kontsevich}
M. Kontsevich, Vassiliev's knot invariants, Adv. in
Sov. Math., 16(2) (1993) 137--150.

\bibitem{Lazard}
M. Lazard,
Groupes analytiques $p$-adiques,
Publ. Math. IHES, {\bf 26}, 5--219, (1965).


\bibitem{Thang Le}
T.Q.T. Le,
On denominators of the Kontsevich integral and the universal perturbative invariant of $3$-manifolds.
Invent. Math. 135 (1999), no. 3, 689--722.

\bibitem{Le Murakami}
T.Q.T. Le and J. Murakami, The universal
Vassiliev-Kontsevich invariant for framed oriented
links, Compositio Math. 102 (1996), 41--64.

\bibitem{Masha} M. Podkopaeva,
On the Jacobson element and generators of the Lie algebra
$\mathfrak{grt}$ in nonzero characteristic,
preprint  arXiv:0812.0772.

\bibitem{Tamarkin}
D. Tamarkin, B. Tsygan,
Noncommutative differential calculus, homotopy BV algebras and formality conjectures,
Methods Funct. Anal. Topology  {\bf 6},  no. 2, 85--100, (2000).

\bibitem{bourbaki}
C. Torossian, La conjecture de Kashiwara-Vergne (d'apr\`es Alekseev et Meinrenken). S\'eminaire Bourbaki. Vol. 2006/2007.  Ast\'erisque  No. 317  (2008), Exp. No. 980, 441--465.

\bibitem{SW}
P.  \v Severa, T. Willwacher,
Equivalence of formalities of the little discs operad,
preprint arXiv:0905.1789.
\end{thebibliography}
\end{document}